\documentclass[a4paper]{article}

\usepackage[utf8]{inputenc}
\usepackage{lmodern}
\usepackage[T1]{fontenc}

\usepackage{mathtools, amssymb, amsthm}
\usepackage{enumitem}
\usepackage{bbm,dsfont}
\numberwithin{equation}{section}

\DeclareMathOperator{\supp}{supp}
\newcommand{\MR}{\textit{MR}}
\newcommand{\CP}{\!\textit{CP}}

\newcommand{\K}{\mathds{K}}
\newcommand{\R}{\mathds{R}}
\newcommand{\C}{\mathds{C}}
\newcommand{\N}{\mathds{N}}
\newcommand{\A}{\mathcal{A}}
\newcommand{\B}{\mathcal{B}}
\newcommand{\Conv}{\mathcal{C}}

\renewcommand{\L}{\mathcal{L}}
\renewcommand{\d}{\mathrm{d}}
\renewcommand{\Re}{\operatorname{Re}}

\newcommand{\fra}{\mathfrak{a}}
\newcommand{\frb}{\mathfrak{b}}
\newcommand{\frc}{\mathfrak{c}}
\renewcommand{\mid}{\, \vert \,}

\DeclarePairedDelimiter\abs{\lvert}{\rvert}
 \DeclarePairedDelimiter\norm{\lVert}{\rVert} 

\theoremstyle{plain}
\newtheorem{theorem}{Theorem}[section]
\newtheorem{proposition}[theorem]{Proposition}
\newtheorem{lemma}[theorem]{Lemma}

\theoremstyle{definition}
\newtheorem{definition}[theorem]{Definition}

\theoremstyle{remark}
\newtheorem{remark}[theorem]{Remark}

\begin{document}
\title{Invariance of  Convex Sets for Non-autonomous Evolution Equations Governed by Forms}
\author{
    Wolfgang Arendt,
    Dominik Dier,
    El Maati Ouhabaz\footnote{Corresponding author.}
}



\maketitle

\begin{abstract}\label{abstract}
We consider a non-autonomous form $\fra:[0,T]\times V\times V \to \C$
where $V$ is a Hilbert space which is densely and continuously embedded in another Hilbert space $H$.
Denote by $\A(t) \in \L(V,V')$ the associated operator. Given  $f \in L^2(0,T, V')$,  
one knows that for each $u_0 \in H$ there is a unique solution $u\in H^1(0,T;V')\cap L^2(0,T;V)$ of
$$\dot u(t) + \A(t) u(t) = f(t),  \, \, u(0) = u_0.$$
This result by J.~L.~Lions is well-known.
The aim of this article is to find a criterion for the invariance of a closed convex subset $\Conv$ of $H$;
i.e.\ we give a criterion on the form which implies that $u(t)\in \Conv$ for all $t\in[0,T]$ whenever $u_0\in\Conv$.
In the autonomous case for $f = 0$,  the criterion is known and even equivalent to invariance  by a result  proved  in \cite{Ouh96} (see also \cite{Ouh05}).
We give applications to positivity and comparison of solutions to heat equations with non-autonomous Robin boundary conditions.
We also prove positivity of the solution to a quasi-linear heat equation. 

\end{abstract}

\medskip


\section{Introduction}\label{section:introduction}

The aim of this article is to prove an invariance criterion for evolution equations governed by a non-autonomous form.
Throughout the article we consider the following situation.
Let $V,H$ be Hilbert spaces over $\K=\R$ or $\C$ such that $V \stackrel d \hookrightarrow H$; i.e., $V$ is densely and continuously embedded in $H$. 
For $T > 0$, we consider  
\[\fra:[0,T]\times V\times V \to \K\]
 such that $\fra(t, ., .)$ is sesquilinear on $V$ 
and satisfies  the following conditions
\begin{align}
	\label{eq:V_bounded}
	&\abs{\fra(t,u,v)} \le M \norm u_V \norm v_V \quad &(t\in [0,T], u,v \in V)\\
	\label{eq:quasi-coercive}
	&\Re \fra(t,u,u) + \omega \norm u_H^2 \ge \alpha \norm u _V^2 \quad &(t\in[0,T], u \in V)\\
	\label{eq:measurable}
	&\fra(.,u,v) \text{ is measurable for all } u,v \in V
\end{align}
In the case where $\K=\R$ this means that $\fra(t,.,.)$ is bilinear;
moreover the real part sign $\Re$ can be omitted in \eqref{eq:quasi-coercive}
and everywhere else in the sequel.
Condition $\eqref{eq:V_bounded}$ means that $\fra(t,.,.)$ is \emph{V-bounded}  with $t$-independent bound,
we call condition $\eqref{eq:quasi-coercive}$ \emph{quasi-coercivity} and   simply \emph{coercivity} if $\omega =0$.
We call such $\fra$, satisfying $\eqref{eq:V_bounded}$-$\eqref{eq:measurable}$ simply a 
\emph{non-autonomous closed form} on $H$.
Define $\A(t) \in \L(V,V')$ by $\langle \A(t)u,v \rangle = \fra(t,u,v)$.
Let $\MR(V,V') := H^1(0,T;V') \cap L^2(0,T;V)$ be the usual maximal regularity space. 
By a theorem due to Lions \cite{Lio61} (see also \cite{Sho97}, Chap. III)  for each $u_0 \in H$, $f\in L^2(0,T;V')$ there exists a unique
$u \in \MR(V,V') $ satisfying 
\begin{equation}\label{eq:solution_of_CP}
\left\{
\begin{aligned}
	&\dot u(t) + \A(t)u(t)=f(t) \quad t\text{-a.e.}\\
	&u(0)=u_0.
\end{aligned}
\right.
\end{equation}
It is well known that $\MR(V,V') \subset C([0,T];H)$ and hence $u$ has a unique continuous representative.
Thus the initial condition makes  sense.\\
Let $\Conv \subset H$ be a closed convex set and denote by $P:H \to \Conv$ the orthogonal projection onto $\Conv$.
Our main result,  Theorem~\ref{thm:invariance},  says the following:
If $u_0 \in \Conv$ and 
\begin{equation}\label{eq:invariance}
P(V) \subseteq V, \hspace{.2cm}  	\Re \fra(t, Pv, v-Pv) \ge  \Re \langle{f(t)}, {v-Pv}\rangle
\end{equation}
for all $v \in V$, then $u(t)\in\Conv$ for $t \in [0,T]$.
For $f \neq 0$ this criterion seems to be new even in the autonomous case.
If $f=0$, then in the autonomous case condition $\eqref{eq:invariance}$ is also necessary for the invariance of $\Conv$. 
The criterion in this autonomous setting is due to \cite{Ouh96} and it is widely used to study positivity, $L^p-$contractivity and  domination for various semigroups. This criterion  is in the spirit of the famous Beurling-Deny criteria which characterize the sub-Markovian property of a semigroup in terms of the corresponding  form. As a corollary of our result,  if we choose the convex set to be the positive cone we characterize positivity of the solution $u$ of the Cauchy problem \eqref{eq:solution_of_CP} if  the initial date $u_0$ and the non-homogeneous term $f$ are positive. 
This corollary is also stated in  \cite[Chap.\ XVIII, § 5]{DL88} however with an erroneous proof. 
Other corollaries  concern a characterization of the sub-markovian property of the solution $u$ as well as comparison of solutions 
$u$ and $v$ of two different Cauchy problems, see Section \ref{section:positivity}.  
In Section~\ref{section:applications} some concrete examples are given.
We consider non-autonomous Robin boundary conditions and also parabolic equations with time dependent coefficients.  In this concrete setting we prove positivity and characterize comparison. We also consider  a quasi-linear problem for which we prove existence of a positive solution.

\subsection*{Acknowledgment}
The authors obtained diverse financial support which they gratefully acknowledge: 
D.\ Dier is a member of the DFG Graduate School 1100: Modelling, Analysis and Simulation in Economathematics, 
E.\ M.\ Ouhabaz visited the University of Ulm in the framework of the Graduate School: 
Mathematical Analysis of Evolution, Information and Complexity financed by the Land Baden-Württemberg and 
W.\ Arendt and D.\ Dier enjoyed a wonderful research stay at the University of Bordeaux.\\
The research of E.\ M.\ Ouhabaz  is partly supported by the ANR project  "Harmonic Analysis at its Boundaries",  ANR-12-BS01-0013-02.

\section{Invariance of closed convex sets}\label{section:invariance}

Let $V$ and $H$ be separable Hilbert spaces over $K=\R$ or $\C$ such that $V \stackrel
d \hookrightarrow H$ (by this we mean that $V$ is continuously and densely embedded in $H$). We keep the same notation as in the introduction and use 
$(. \mid .)_H$, $(. \mid .)_V$, $\norm{.}_H$ and $\norm{.}_V$ for their scalar products  and norms and denote by $\langle ., . \rangle$
the duality between $V'$ and $V$.  We consider a non-autonomous closed form
    \[
        \fra:[0,T]\times V \times V \to \K.
    \]
As in the introduction we denote by $\A(t) \in
\L(V,V')$ the operator associated with the form $\fra$. Given $f \in
L^2(0,T;V')$ and $u_0 \in H$ we have seen in $\eqref{eq:solution_of_CP}$ that the Cauchy problem
\begin{equation*}
    (\CP_f) \left\{ 
	\begin{aligned} 
		&\dot u(t) + \A(t)u(t)= f(t)\quad \text{a.e.}\\ 
		&u(0) = u_0 
	\end{aligned} \right.
\end{equation*}
has a unique solution $u \in \MR(V,V') = H^1(0,T;V') \cap
L^2(0,T;V)$.

In this section we study invariance properties of the solution
$u$. To make this precise, let us fix a closed convex set
$\Conv$ of $H$. We introduce the following definition.

\begin{definition} Given $f\in L^2(0,T;V')$,  we say that the
     convex set $\Conv$ is \emph{invariant} for the Cauchy problem $(\CP_f)$
     if for each $u_0\in \Conv$ the solution $u$ of $(\CP_f)$ satisfies
     $u(t)\in \Conv$ for all $t \in [0,T]$.
   \end{definition}

Recall that the solution of $(\CP_f)$ is in $C([0,T];H)$.
Our aim is to provide a criterion in terms of $\fra(t, . , .)$
and $f$ for this invariance property. As an application we obtain
positivity and $L^p$-contractivity for the Cauchy problem
$(\CP_f)$. In the autonomous case, criteria in
terms of the form that characterize this property are given in 
\cite{Ouh96}.

Let $P: H \to \Conv$ be the orthogonal projection onto $\Conv$;
i.e.\ for $x\in H$, $Px$ is the unique element $x_\Conv$ in $\Conv$ such
that
\begin{equation}\label{eq:characterization_orthogonal_projection}
	\Re ({x-x_\Conv} \mid {y-x_\Conv} )_H \le 0
\end{equation}
for all $y\in \Conv$.
In the  autonomous case the invariance of $V$ under $P$ is a necessary
condition for invariance of $\Conv$ \cite[Theorem 2.1]{Ouh96}. Thus we assume throughout
that
\begin{equation}\label{eq:V_invariant}
    P(V) \subset V.
\end{equation}
Our main result in this section is the following.

\begin{theorem}\label{thm:invariance}
    Let $\fra$ be a  non-autonomous closed  form on $V$.
    Let $f \in L^2(0,T;V')$ and let $\Conv$ be a closed convex subset of $H$ such that $\eqref{eq:V_invariant}$
    is satisfied. Then $\Conv$ is invariant for $(\CP_f)$ provided
    \begin{equation}\label{eq:invariance_criterium}
        \Re \fra(t, Pv, v-Pv) \ge  \Re \langle{f(t)}, {v-Pv}\rangle
    \end{equation}
    for all $v \in V$ and a.e.\ $t \in [0,T]$.
\end{theorem}

The following lemma is crucial for the proof of Theorem
\ref{thm:invariance}.

\begin{lemma}\label{lem:invariance}
    Let $ u \in \MR(V,V')$. Then for  $t,r \in [0,T]$ with $r \le t$ the following identity holds:
    \[
        \norm{u(t)- Pu(t)}_H^2 - \norm{u(r)-Pu(r)}_H^2
            = 2 \int_r^t \Re \langle{\dot u(s)}, {u(s)-Pu(s)}\rangle \ \d s
    \]
\end{lemma}
\begin{proof}
    We consider for simplicity the case $r=0$ and $t=T$.
    Recall that
    \begin{equation}\label{eq:Cinfty_dense_in_MR}
        C^\infty ([0,T];V) \text{ is dense in } \MR(V,V')
    \end{equation}
    and
    \begin{equation}\label{eq:MR_cont_emb_in_C}
        \MR(V,V') \hookrightarrow C([0,T];H).
    \end{equation}
    For these two properties, see e.g.\ \cite[Proposition III.1.2]{Sho97}. 
    By $\eqref{eq:Cinfty_dense_in_MR}$, there exists a sequence
    $(u_n)_{n \in \N} \subset C^\infty ([0,T];V)$ such that
    \begin{equation}\label{eq:convergency_in_MR}
            u_n \to u \text{ in } L^2(0,T;V) \text{ and } \dot u_n \to \dot u \text{ in } L^2(0,T;V').
    \end{equation}
    For each fixed $n$,
    \[
        \norm{u_n(T)- Pu_n(T)}_H^2 - \norm{u_n(0)-Pu_n(0)}_H^2
            = \int_0^T \frac{\d}{\d s} \norm{u_n(s)- Pu_n(s)}_H^2 \ \d s.
    \]
    Note that $P:H \to H$ is a contraction and hence $Pu_n \in H^1(0,T;H)$
    (see Theorem \ref{thm:appendix} in the Appendix). 
    Thus
    \begin{align*}
        &\norm{u_n(T)- Pu_n(T)}_H^2 - \norm{u_n(0)-Pu_n(0)}_H^2\\
            &\hspace{1cm}= 2 \int_0^T \Re ({\dot u_n(s)-(Pu_n)\dot{}\,(s)} \mid {u_n(s)-Pu_n(s)})_H \ \d s.
    \end{align*}
    Now for a.e.\ $s \in (0,T)$
    \begin{align*}
         &\Re ({(Pu_n)\dot{}\,(s)} \mid {u_n(s)-Pu_n(s)})_H \\
            &\hspace{1cm}= \lim_{h \to 0} \Re \tfrac 1 h ({Pu_n(s+h)-Pu_n(s)} \mid {u_n(s)-Pu_n(s)})_H
    \end{align*}
    Using $\eqref{eq:characterization_orthogonal_projection}$ we see that
    the right hand side is positive  for $h > 0$ and
    negative  for $h < 0$.
    Thus
    \[
         \Re ({(Pu_n)\dot{}\,(s)} \mid {u_n(s)-Pu_n(s)} )_H = 0.
    \]
    It follows that
    \begin{equation}\label{eq:lem_invariance}
        \begin{aligned}
            \norm{u_n(T)- &Pu_n(T)}_H^2 - \norm{u_n(0)-Pu_n(0)}_H^2\\
                &= 2 \int_0^T \Re ({\dot u_n(s)} \mid {u_n(s)-Pu_n(s)})_H \ \d s\\
                &= 2 \int_0^T \Re \langle{\dot u_n(s)}, {u_n(s)-Pu_n(s)}\rangle \ \d s.
        \end{aligned}
    \end{equation}
    By $\eqref{eq:MR_cont_emb_in_C}$ and the continuity of $P$ in $H$ it follows that
    the left hand side of $\eqref{eq:lem_invariance}$ converges to
    $\norm{u(T)- Pu(T)}_H^2 - \norm{u(0)-Pu(0)}_H^2$.

    Suppose $u_n-Pu_n$ is bounded in $L^2(0,T;V)$,
    then there exists a subsequence converging weakly to some function $g \in L^2(0,T;V)$.
    Moreover since $P:H \to H$ is a contraction and since $u_n \to u$ in $L^2(0,T;V)$ 
    also $u_n-Pu_n$ converges to $u-Pu$ in $L^2(0,T;H)$.
    Thus $g=u-Pu$ and by $\eqref{eq:convergency_in_MR}$
    a subsequence of the right hand side of $\eqref{eq:lem_invariance}$ converges to
    $2 \int_0^T \Re \langle{\dot u(s)}, {u(s)-Pu(s)}\rangle \, \d s$.

    Thus to prove the lemma it remains to show that $u_n-Pu_n$ is bounded in $L^2(0,T;V)$.
    By $\eqref{eq:invariance_criterium}$ we have
    \[
        \Re \fra(t, Pu_n(t), u_n(t)-Pu_n(t)) \ge  \Re \langle{f(t)}, {u_n(t)-Pu_n(t)}\rangle
    \]
    Hence by quasi-coercivity and $V$-boundedness of the form $\fra$
    \begin{align*}
        \alpha \norm{u_n&(t)-Pu_n(t)}_V^2\\
            &\le \Re \fra(t,u_n(t)-Pu_n(t),u_n(t)-Pu_n(t))
               + \omega \norm{u_n(t)-Pu_n(t)}_H^2\\
            &\le \Re \fra(t,u_n(t),u_n(t)-Pu_n(t))
                - \Re \langle{f(t)}, {u_n(t)-Pu_n(t)}\rangle\\
                &\quad + \omega \norm{u_n(t)-Pu_n(t)}_H^2\\
            &\le M \norm{u_n(t)}_V \norm{u_n(t)-Pu_n(t))}_V
                 + \norm{f(t)}_{V'} \norm{u_n(t)-Pu_n(t)}_{V}\\
                &\quad+ \omega \norm{u_n(t)-Pu_n(t)}_H^2.
    \end{align*}
    From this and the standard inequality
    \[
        a b \le \frac 1 {4\epsilon} a^2 + \epsilon b^2 \quad (\epsilon >0, a,b \in \R)
    \]
    we see that for some constant $M' >0$ (independent of $n$ and $t$)
    \begin{align}\label{eq:V_norm_estimate}
        \norm{u_n(t&)-Pu_n(t)}_V^2 \le
            M' \left( \norm{u_n(t)}^2_V + \norm{f(t)}^2_{V'} + \norm{u_n(t)-Pu_n(t)}_H^2 \right).
    \end{align}
    By $\eqref{eq:convergency_in_MR}$ the sequence $(u_n)_{n \in \N}$ is bounded in $L^2(0,T;V)$.
    Consequently,  it is bounded in $L^2(0,T;H)$. 
    Since $P$ is a contraction on $H$ the sequence $(u_n-Pu_n)_{n \in \N}$ is bounded in $L^2(0,T;H)$.
    It follows from $\eqref{eq:V_norm_estimate}$ that $(u_n-Pu_n)_{n \in \N}$ is a bounded sequence in $L^2(0,T;V)$.
\end{proof}

\begin{proof}[Proof of Theorem \ref{thm:invariance}]
    Fix $u_0 \in \Conv$. Our aim is to prove that
    $u(t) \in \Conv$ for all $t\in [0,T]$.
    By Lemma \ref{lem:invariance} and $\eqref{eq:V_invariant}$,  for all $t \in [0, T]$ we have 
    \begin{align*}
        \norm{u&(t)- Pu(t)}_H^2\\
            &= 2 \int_0^t \Re \langle{\dot u(s)}, {u(s)-Pu(s)} \rangle \ \d s\\
            &= 2 \int_0^t \Re \langle{-\A(s) u(s)+f(s)}, {u(s)-Pu(s)} \ \rangle \ \d s\\
            &= 2 \int_0^t \big[ - \Re \fra(s,u(s),u(s)-Pu(s))
                 + \Re \langle{f(s)}, {u(s)-Pu(s)} \rangle \big]  \ \d s\\
            &= 2 \int_0^t \big[ - \Re \fra(s,u(s)-Pu(s),u(s)-Pu(s)) \\
                &\quad - \Re \fra(s, Pu(s),u(s)-Pu(s)) + \Re \langle{f(s)}, {u(s)-Pu(s)}\rangle \big] \ \d s\\
            &\le 2 \int_0^t - \Re \fra(s,u(s)-Pu(s),u(s)-Pu(s)) \ \d s
    \end{align*}
    where we used the assumption $\eqref{eq:invariance_criterium}$ for  the last inequality.
    From quasi-coercivity of the form $\fra$ and the previous estimate we obtain
    \[
        \norm{u(t)- Pu(t)}_H^2 \le 2 \omega \int_0^t \norm{u(s)-Pu(s)}_H^2 \ \d s \quad (t\in[0,T]).
    \]
    We conclude by Gronwall's lemma that $\norm{u(t)- Pu(t)}_H^2 =0$ for all $t \in [0,T]$.
\end{proof}

The following extension  of Theorem  \ref{thm:invariance} is of interest in applications.  Let $\fra$ be as  before and  $f \in L^2(0, T, V')$.  Consider two closed convex sets 
  $\Conv_1$ and $ \Conv_2$ of $H$ and denote by 
$P_1$ and $P_2$ the orthogonal  projections onto $\Conv_1$ and $ \Conv_2$, respectively. 

\begin{theorem}\label{thm:invariance2} 
Suppose that $\Conv_1$ is invariant for $(CP_f)$ and that 
$$v \in V \cap \Conv_1 \ {\rm  implies}\  P_2 v \in V \ {\rm and}$$
\begin{equation}\label{1,2}
        \Re \fra(t, P_2v, v-P_2v) \ge  \Re \langle{f(t)}, {v-P_2v}\rangle
    \end{equation}
    for a.e.\ $t \in [0,T]$.
Then $\Conv_1 \cap \Conv_2$ is invariant for $(CP_f)$.
\end{theorem}

\begin{proof} Suppose  that $u_0 \in \Conv_1 \cap \Conv_2$. Since $\Conv_1$ is invariant for $(CP_f)$,  the solution $u$ satisfies  
$$ u(t) \in \Conv_1 \cap V \ {\rm for } \ a.e. t \in [0, T].$$
We extend $u$ to  $\tilde{u}$ on $\R$  by $u(0)$ on $(-\infty, 0)$ and $u(T)$ on $(T, \infty)$. Note that $u(0)$ and $u(T) \in V$ because of 
\eqref{eq:MR_cont_emb_in_C}. 
Take the  convolution $u_n  := \rho_n \star \tilde{u}$ with  a standard mollifier $\rho_n$. 
Then $u_n \in C^\infty (\R; V)$ and $u_n(t) \in \Conv_1$ for all $t \in \R$
and \eqref{1,2} holds for $v= u_n(t)$ for each $n$ and all $t \in [0,T]$. Using this sequence we obtain exactly as in the proof of Lemma \ref{lem:invariance} 
\[
        \norm{u(t)- P_2u(t)}_H^2 - \norm{u(r)-P_2u(r)}_H^2
            = 2 \int_r^t \Re \langle{\dot u(s)}, {u(s)-P_2u(s)}\rangle \ \d s.
    \]
 From this we can reproduce the  proof of Theorem \ref{thm:invariance}. 
  Indeed, 
 \begin{align*}
        \norm{u&(t)- P_2u(t)}_H^2\\
            &= 2 \int_0^t \Re \langle{\dot u(s)}, {u(s)-P_2u(s)} \rangle \ \d s\\
            &= 2 \int_0^t \big[ - \Re \fra(s,u(s),u(s)-P_2u(s))
                + \Re \langle{f(s)}, {u(s)-P_2u(s)} \rangle \big]  \ \d s\\
            &= 2 \int_0^t \big[ - \Re \fra(s,u(s)-P_2u(s),u(s)-P_2u(s)) \\
                &\quad - \Re \fra(s, P_2u(s),u(s)-P_2u(s)) + \Re \langle{f(s)}, {u(s)-P_2u(s)}\rangle \big] \ \d s\\
            &\le 2 \int_0^t - \Re \fra(s,u(s)-P_2u(s),u(s)-P_2u(s)) \ \d s. 
    \end{align*}
   where we use (\ref{1,2}) since $u(s) \in \Conv_1 \cap V$ for a.e.\ $s$. We use again quasi-coercivity and Gronwall's lemma to obtain
    $\norm{u(t)- P_2u(t)}_H^2 =0$ for all $t \in [0,T]$.
\end{proof}

\section{Positivity and comparaison}\label{section:positivity}

In this section we assume for simplicity that $\K=\R$ and
let  $H=L^2(\Omega,\mu)$ where $(\Omega, \mu)$ is a measure space.
For $f \in L^2(\Omega,\mu)$ we let $f^+(x) := \max\{ f(x),0 \}$, $f^-:=(-f)^+$, $\abs f = f^+ + f^-$.
We write $f \ge 0$ as short hand for $f(x) \ge 0$ $\mu$-a.e.
We keep the notations of the introduction; i.e.\ $V$ is a Hilbert space which is continuously and dense imbedded  into $H=L^2(\Omega,\mu)$
and $\fra : [0,T]\times V\times V \to \R$ is a non-autonomous closed form.
We say that $V$ is a \emph{sublattice} of $H$ if $v \in V$ implies $v^+ \in V$.
We let $L^2(\Omega, \mu)_+ := \{ g \in L^2(\Omega, \mu) : g \ge 0 \}$ and
$V_+ := L^2(\Omega, \mu)_+ \cap V$.
Given $f \in L^2(0,T;V')$, we say that $f$ is \emph{positive} and write $f \ge 0$ if $\langle f(t), v \rangle \ge 0$ $t$-a.e.\ for all $0 \le v \in V$.

\begin{proposition}\label{cor:positivity}
	Assume that $V$ is a sublattice of $H$ and $\fra(t,v^+,v^-) \le 0$ for a.e.\ $t$ and all $v\in V$.
	Let $u_0 \in V_+$ and $f \ge 0$.
	Then the solution $u$ of \eqref{eq:solution_of_CP} satisfies $u(t) \ge 0$ for all $t\in [0,T]$.
\end{proposition}
\begin{proof}
    We take the closed convex set
    \[
        \Conv = \{ v \in L^2(\Omega, \mu) : v \ge 0 \}
    \]
    The orthogonal projection onto $ \Conv$ is given by $Pv =v^+$.
    By our assumptions we have
    \[
        \fra(t,v^+,v^-) \le 0 \le  \langle {f(t)}, {v^-} \rangle \quad (t \in [0,T]).
    \]
    This implies $\eqref{eq:invariance_criterium}$ and we apply Theorem~\ref{thm:invariance}.
\end{proof}

This proposition is known. It is formulated in Theorem 2 of
\cite[Chap.\ XVIII, § 5]{DL88}. However the proof there seems not
correct (one cannot take $v=-u^-$ in (4.49) since $u$ depends on
$t$). A correct proof is given in \cite{Tho03}.  Corollary \ref{cor:positivity} is also proved in the case of
elliptic operators with Dirichlet or Neumann boundary conditions
in \cite{DD97}.  In the autonomous case the criterion is also necessary,
see \cite{Ouh96}.

Next we consider the submarkovian property. 
For  $v \in L^2(\Omega,\mu)$, we set $v\wedge 1 = \inf\{v, 1\}$.  Then, $v\wedge 1, (v-1)^+ \in L^2(\Omega,\mu)$ and $v = v\wedge 1 + (v-1)^+$.
\begin{proposition}\label{sub}
	Assume that $v\wedge 1\in V$ and $\fra(t, v\wedge 1, (v-1)^+) \ge 0$ for all $t\in[0,T]$, $v\in V$ and that $f \le 0$. 
	Let $u_0 \in L^2(\Omega, \mu)$ such that $u_0 \le 1$ $\mu$-a.e.
	Then the solution $u$ of $\eqref{eq:solution_of_CP}$ satisfies $u(t)\le 1$ $\mu$-a.e.\ for all $t\in [0,T]$.
	In particular, if $u_0 \le 0$, then $u(t) \le 0$ for all $t\in [0,T]$.
\end{proposition}
\begin{proof} We choose the convex set
$$\Conv = \{ v \in L^2(\Omega, \mu) : v \le 1 \}.$$
	The orthogonal projection $P$ from $L^2(\Omega,\mu)$ to $\Conv$ is given by $Pv = v \wedge 1$.
	Moreover, for $v\in V$,
	\[
		\fra(t, Pv, v-Pv) = \fra(t, v\wedge 1, (v-1)^+)  \ge 0 \ge \langle f, (v-1)^+ \rangle.
	\]
	Thus the first claim follows from Theorem~\ref{thm:invariance}.

	From homogeneity it follows that $u_0 \le \lambda$ $\mu$-a.e.\ implies $u(t) \le \lambda$ $\mu$-a.e.\ for all $\lambda >0$.
	If $u_0 \le 0$, it follows that $u(t)\le \lambda$ for all $\lambda >0$.
	Hence $u(t) \le 0$.
	Applying this to $-u_0$ instead of $u_0$ the claim follows.
\end{proof}

Next we investigate domination. For that we consider a second Hilbert space $W \stackrel d \hookrightarrow L^2(\Omega,\mu)$ 
and a closed non-autonomous form $\frb:[0,T]\times W \times W \to \R$.
We denote by $\B(t) \in \L(W,W')$ the operator given by $\langle \B(t)w, v\rangle = \frb(t,w,v)$ for $w,v \in W$.
We consider $L^2(\Omega,\mu) \hookrightarrow W'$ as before.
Then for all $w_0 \in W$ there exists a unique $w \in \MR(W,W')$ satisfying
\begin{equation}\label{eq:solution_of_CP_b}
\left\{
\begin{aligned}
	&\dot v(t) + \B(t)v(t)=g(t)& \quad t\text{-a.e.}\\
	&	v(0)=v_0.
\end{aligned}
\right.
\end{equation}
\begin{proposition}\label{dom}
	Assume that $V, W$ are sublattices of $H$ such that $V$ is an {\it ideal} of $W$ in the sense that $V \subset W$ and for $v \in V$, $w \in W$, 
	 $0\le w \le v$ implies $w\in V$. We  assume  furthermore that $\fra(t, v^+,v^-)\le 0$ for all $v\in V$ and $\frb(t, w^+,w^-)\le 0$ for all $w \in W$.
	Let $u_0 \in H_+$ and $v_0 \in W$ such that $u_0 \le v_0$ and $f,g\in L^2(0,T;V')$ such that $f\le g$.
	Then the solution $u$ of $\eqref{eq:solution_of_CP}$ and the solution $v$ of $\eqref{eq:solution_of_CP_b}$ satisfy 
\[
		u(t) \le v(t) \quad (t\in [0,T]),
\]
provided the forms satisfy  $\frb(t,u,v) \le \fra(t,u,v)$ for $a.e.\ t\in [0,T]$ and all $u,v\in V_+$.
\end{proposition}
 \begin{proof}
 We consider the space $L^2(\Omega, \mu) \times L^2(\Omega,\mu)$ and the form
	\[
		\frc: [0,T]\times(V\times W) \times (V\times W) \to \R
	\]
	given by $\frc(t,(v_1,w_1),(v_2,v_2)) = \fra(t,v_1,v_2) + \frb(t,w_1,w_2)$.
	Consider the set
	\[
		\Conv := \{ (u,v)\in L^2(\Omega, \mu) \times L^2(\Omega,\mu) \colon  0 \le u \le v \}.
	\]
	Then $\Conv$ is closed and convex. The conclusion of the proposition follows from invariance of $\Conv$.\\
	Note that by assumptions and Proposition \ref{cor:positivity}, the convex set 
	\[ 
	\Conv_1 := \{ (u,v)\in L^2(\Omega, \mu) \times L^2(\Omega,\mu) \colon  0 \le u \ {\rm and} \  0 \le v \}
	\]
	is invariant. Hence by Theorem~\ref{thm:invariance2} we may restrict  attention to non-negative $u \in V$ and $v \in W$.
	The projection $P$ from $L^2(\Omega, \mu) \times L^2(\Omega,\mu)$ to $\Conv$ is given by
	$$P(u,v)=(u-\tfrac 1 2 (u-v)^+, v + \tfrac 1 2 (u-v)^+).$$
	Suppose now that $(v,w) \in V \times W$  with  $0 \le v$ and $0 \le v$. Set $u:=\tfrac 1 2 (v-w)^+$.
	We have $(u,v)-P(u,v)=(u,-u)$, $0\le v-u \le v$ and $0 \le u \le v$.
	Thus by assumption b) $v-u, u \in V$ and
	\begin{align*}
		\frc(t,P&(v,w),(v,w)-P(v,w))\\
			&= \fra(t, v-u, u) + \frb(t, w +u, -u)\\
			&= \fra(t, v-u, u) - \frb(t, v-u, u) + \frb(t,v-u - (w+u) ,u)\\
			&\ge 0\\
			& \ge  \langle f, u \rangle -\langle g, u \rangle\\
			&= \langle (f,g), (v,w)-P(v,w) \rangle
	\end{align*}
	by assumption d) and the inequality
	\[
		\frb(t,v-u - (w+u) ,u) = -\tfrac 1 2 \frb(t,(w-v)^+ , (w-v)^-) \ge 0
	\]
	 for a.e.\ $t\in [0,T]$. Now the claim follows from Theorem~\ref{thm:invariance2}.
\end{proof}

\section{Applications}\label{section:applications}
In this section we give some   applications to concrete   examples which illustrate our abstract results. In all cases, we deliberately consider typical, simple situations and do not aim for the greatest generality. \\

{\it I)  Elliptic operators with time-dependent coefficients.}
We consider elliptic operators of second order with time-dependent coefficients.  Let $\Omega$ be an open set of $\R^d$ and  consider on the real Hilbert space $L^2(\Omega,dx)$  the form
$$\fra(t,u,v) = \sum_{k,j=1}^d \int_\Omega a_{kj}(t,x) \partial_k u \partial_j v \ \d x$$
for $u, v \in V$ where $V$ is a closed subspace of $H^1(\Omega)$ which contains $H_0^1(\Omega)$.  Recall that $H^1(\Omega)$ and  $H_0^1(\Omega)$ are sublattices of 
$L^2(\Omega)$. We assume that the coefficients $a_{kj}$ are measurable and uniformly bounded on $[0,T] \times \Omega$ and satisfy the usual ellipticity condition
$$\sum_{k,j=1}^d a_{kj}(t,x) \xi_k \xi_j \ge \eta |\xi |^2$$
for  all $\xi = (\xi_1, ..., \xi_d) \in \R^d$ and for a.e. $(t,x) \in [0, T] \times \Omega$. Here $\eta > 0$ is a constant.
As before we denote by $\A(t)$ the associated operator and for given $f \in L^2(0,T; V')$ we denote  by $u \in \MR(V,V')$ the solution of the Cauchy problem \eqref{eq:solution_of_CP}.
\begin{proposition}\label{pos}
1) Suppose that $f \ge 0$ and $v^+ \in V$ for all $v \in V$. If $u_0 \ge 0$ then  $u(t) \ge 0$ for all $t \in [0, T]$.\\
2) Suppose that $f \le 0$ and $v\wedge 1 \in V$ for all $v \in V$. If $u_0 \le 1$ then $u(t) \le 1$ for all $t \in [0,T]$.
\end{proposition}

\begin{proof} This is an immediate application of  Propositions \ref{cor:positivity} and \ref{sub} and the classical formulae
$$\partial_k v^+ = \chi_{\{ v > 0\}} \partial_k v, \ \partial_k (v\wedge 1) = \chi_{v \le 1} \partial_k v$$
for all $v \in H^1(\Omega)$. 
\end{proof}

The space $V$ incorporates the boundary condition. For example, if $V=H^1_0(\Omega)$, 
then we deal with Dirichlet boundary conditions whereas $V=H^1(\Omega)$ corresponds 
to Neumann boundary conditions in the sense that the conormal derivative $\frac{\partial u}{\partial\nu_A}$
(depending on the coefficients $a_{kj}$) vanishes at the boundary.
Since for $u \in H^1(\Omega)$, $v \in H^1_0(\Omega)$, $0 \le u \le v$ implies that $u \in H^1_0(\Omega)$
we deduce from Proposition~\ref{dom} the following.
\begin{proposition}
Suppose that $f\ge0$ and let $u_0\in L^2(\Omega)_+$. 
Denote by $\underline u \in \MR(H^1_0(\Omega), H^1_0(\Omega)')$
the solution of \eqref{eq:solution_of_CP} with respect to $V=H^1_0(\Omega)$
and by $\overline u \in \MR(H^1(\Omega), H^1(\Omega)')$
the solution of \eqref{eq:solution_of_CP} with respect to $V=H^1(\Omega)$.
Then $0 \le \underline u(t) \le \overline u(t)$ for all $t \in [0,T]$.
\end{proposition}

{\it II)  A quasi-linear problem.}  In our second example we consider a quasi-linear problem for which we prove existence of a positive solution. \\
Let $\Omega$ be a bounded open set of $\R^d$ and let $H$ be the real-valued Hilbert space $L^2(\Omega, \d x)$ and let $V$ be a closed subspace of $H^1(\Omega)$ which contains $H_0^1(\Omega)$. If $V \not= H^1_0(\Omega)$ we assume that $\Omega$
has continuous boundary (in the sense of graphs) 
to ensure that the embedding 
of  $V$ in $H$ is compact.  This latter property is always true for  $V=H^1_0(\Omega)$ on any bounded domain $\Omega$. \\
For $j,k \in \{ 1, \dots, d \}$ let
\[ m_{kj} : [0,T] \times \Omega \times \R \to \R \]
be measurable functions such that $m_{kj}(t,x,.)$ is continuous for a.e.\ $(t,x)$. We assume furthermore  that there exists $\eta > 0$ such that
\[ 
	\sum_{k,j=1}^d m_{kj}(t,x,y) \xi_k \xi_j  \ge \eta |\xi |^2
\]
for all $t\in [0,T]$, $x\in \Omega$, $y \in \R$ and $\xi \in \R^d$.
Finally  we assume that $\abs{m_{kj}(t,x,y)} \le const$ for all 
$t\in [0,T]$, $x\in \Omega$, $y \in \R$.
The quasi-linear problem we consider here is the following:
\begin{equation}\label{eq:quasi_linear_problem}
(NCP)\left\{
\begin{aligned}
	&\dot u - \sum_{k,j=1}^d \partial_k(m_{kj}(t, x, u) \partial_j u)  = f(t) \quad t\text{-a.e.}\\
	&	u(0)=u_0 \in H\\
	&	u \in \MR(V, V')
\end{aligned}
\right.
\end{equation}
Here we assume that $f \in L^2(0,T, V')$ and $u_0\in H$ are given.
If $u \in \MR(V,V') \subset L^2(0,T;H)$, then $a_{kj}(t,x):= m_{kj}(t,x,u(t)(x))$
defines measurable functions for $j,k =1,\dots,d$ which satisfy  the assumptions of {I)}.
We denote by $\A_u(t)\in \L(V,V')$ the elliptic  operator with coefficients $a_{kj}(t,x)$.
A function $u \in \MR(V,V')$ is called a \emph{solution} of \eqref{eq:quasi_linear_problem} if
\[ \left\{
\begin{aligned}
&\dot u(t)+\A_u (t)u(t)=f(t) \quad \text{a.e.}\\
&u(0)=u_0.
\end{aligned}
\right.
\]
We shall prove existence of a solution to $(NCP)$ which in addition is non-negative if the initial data $u_0$ is non-negative. 
This will be done by a fixed point argument. 
Given $g \in L^2(0,T;H)$, consider the non-autonomous closed form
$\fra_g:[0,T]\times V \times V \to \R$ defined by 
\[ \fra_g(t,u,v) := \sum_{k,j=1}^d \int_\Omega m_{kj}(t,x,g(t)) \partial_k u \partial_j v \ \d x.\]
Note that we can choose constants such that
\eqref{eq:V_bounded}, \eqref{eq:quasi-coercive} holds 
(taking $\alpha:= \eta$, $\omega := \eta$ and $M$ appropriately). 
We denote by $\A_g(t)$ the operator associated with the form $\fra_g(t, ., .)$.
For $u_0 \in H$ and $f \in L^2(0,T;V')$ we know by Lions' theorem (see \eqref{eq:solution_of_CP}) that there exists a  
 unique solution  $u_g \in \MR(V, V')$ of the Cauchy problem
\[
\left\{
\begin{aligned}
	&\dot u_g(t) + \A_g(t)u_g(t)=f(t) \quad t\text{-a.e.}\\
		&u_g(0)=u_0.
\end{aligned}
\right.
\]
 In addition 
\begin{equation}\label{estV}
\norm{u_g}_{\MR(V,V')} \le C \left[ \norm{u_0}_H + \norm{f}_{L^2(0,T;V')} \right], 
\end{equation}
where $C > 0$ is a constant which does not depend on $g$, see \cite{Sho97}, Proposition 4.12 on p.\ 112 and 
subsequent comments.  We define the mapping $S:L^2(0,T;H) \to L^2(0,T;H)$ by $Sg := u_g$. 
By \eqref{estV}, the range of $S$ is bounded in $\MR(V,V')$. In addition, since we assume that $V$ is compactly embedded  into $H$ 
we deduce  from  the Aubin-Lions Lemma \cite[Proposition III.1.3]{Sho97} that the range of $S$ is relatively compact in $L^2(0,T, H)$. Thus it remains to prove that $S$ is continuous
to conclude  by Schauder's fixed point theorem that there exists  $u \in \MR(V,V')$ such that $Su=u$. This $u$ is a solution of \eqref{eq:quasi_linear_problem}.

Now we show that $S$ is continuous. Let $g_n \to g$ in $L^2(0,T;H)$ and set $u_n := Sg_n$. 
Since a sequence converges to a fixed element $u$ if and only if each subsequence has a subsequence converging to $u$ we may deliberately take subsequences.
Since $L^2(0,T;H)$ is isomorphic to $L^2((0,T)\times \Omega)$ we may assume (after taking a sub-sequence) that $g_n \to g$ for a.e.\ $(t,x)$. Furthermore since the sequence $u_n$ is bounded in $\MR(V,V')$ we may assume (after taking a sub-sequence) that 
$u_n \to u$ in $L^2(0,T;H)$ and $u_n \rightharpoonup u$ in $\MR(V,V')$.
Thus $\dot u_n \rightharpoonup \dot u$ in $L^2(0,T;V')$, $\partial_j u_n \rightharpoonup \partial_j u$ in $L^2(0,T;H)$ for all $j \in \{1, \dots d\}$ and $u_n(0) \to u(0)=u_0$ in $H$ since $\MR(V, V')$ is continuously embedded into $C([0,T], H)$. 
Since $g_n \to g$ for a.e.\ $(t,x)$ also 
$m_{kj}(t,x,g_n(t)(x)) \to m_{kj}(g(t,x,g(t)(x))$ for a.e.\ $(t,x)$ and all $j,k \in \{1, \dots d\}$. Finally $Sg_n = u_n$ is equivalent to
\begin{align*}
	\langle \dot u_n, v \rangle_{L^2(0,T;V'),L^2(0,T;V)} &+ \sum_{j,k=1}^d ( \partial_j u_n \mid m_{jk}(t,x,g_n) \partial_k v )_{L^2(0,T;H)} \\
	=&{} \langle f, v \rangle_{L^2(0,T;V'),L^2(0,T;V)} \quad (v \in L^2(0,T;V))
\end{align*}
and $u_n(0) = u_0$.
Since by the dominated convergence  theorem 
\[
	m_{jk}(t,x,g_n) \partial_k v \to m_{jk}(t,x,g) \partial_k v
\]
in $L^2(0,T;H)$, taking the limit as $n\to \infty$ yields
\begin{align*}
	\langle \dot u, v \rangle_{L^2(0,T;V'),L^2(0,T;V)} &+ \sum_{j,k=1}^d ( \partial_j u \mid m_{jk}(t,x,g) \partial_k v )_{L^2(0,T;H)} \\
	=&{} \langle f, v \rangle_{L^2(0,T;V'),L^2(0,T;V)} \quad (v \in L^2(0,T;V))
\end{align*}
and $u(0) = u_0$, which is equivalent to $Sg=u$. Hence $S$ is continuous and we have existence of a solution $u$. 

In order to prove positivity we observe  that for any $g \in L^2(0,T;H)$ we may apply Proposition~\ref{pos} to $u_g = Sg$ and deduce that each $u_g$ is positive. Consequently, 
also the fixed point $u$ is positive.  
We have proved the following result. 
\begin{proposition}\label{pos_ql}
For $u_0 \in H$ and $f \in L^2(0,T, V')$ there exists a solution $u$ to $(NCP)$. In addition $u$ satisfies the following assertions. \\
1) Suppose that $f \ge 0$ and $v^+ \in V$ for all $v \in V$. If $u_0 \ge 0$ then  $u(t) \ge 0$ for all $t \in [0, T]$.\\
2) Suppose that $f \le 0$ and $v\wedge 1 \in V$ for all $v \in V$. If $u_0 \le 1$ then $u(t) \le 1$ for all $t \in [0,T]$.
\end{proposition}


{\it III) Non-autonomous Robin boundary conditions.} Our third example concerns the Laplacian with time dependent Robin boundary conditions. 
We suppose that $\Omega$ be a  bounded domain of $\R^d$ with Lipschitz boundary $\Gamma$. 
Denote by $\sigma$ be the $(d-1)$-dimensional Hausdorff measure on $\Gamma$.
Let  
\[
	\beta: [0,T]  \times \Gamma \to \R
\]
be a bounded measurable function.   We consider the symmetric form
\[
	\fra: [0,T] \times H^1(\Omega) \times H^1(\Omega) \to \R
\]
defined by 
\begin{equation}\label{formbeta}
	\fra(t, u, v) = \int_\Omega \nabla u \nabla v\ \d x + \int_\Gamma \beta(t, .) u v\ \d\sigma.
\end{equation}
In the second integral we omitted the trace symbol. 
The form $\fra$ is $H^1(\Omega)$-bounded and quasi-coercive. 
The first statement follows readily from the continuity 
of the trace operator and the boundedness of $\beta$. 
The second one is a consequence of the inequality 
\begin{equation}\label{trace-comp}
\int_\Gamma \lvert u \rvert^2 \ \d\sigma \le \epsilon \norm u_{H^1(\Omega)}^2 + c_\epsilon \norm u_{L^2(\Omega)}^2,
\end{equation}
which is valid for all $\epsilon > 0$ ($c_\epsilon$ is a constant depending on $\epsilon$). 
Note that $\eqref{trace-comp}$ is a consequence of compactness of the trace as an operator from $H^1(\Omega)$ into $L^2(\Gamma, \d \sigma)$, see \cite[Chap.\ 2 § 6, Theorem 6.2]{Nec67}.

The operator $A(t)$ associated with $\fra(t,.,.)$ on $H:= L^2(\Omega)$ 
is (minus) the Laplacian  with  time dependent Robin boundary conditions
\[
	\partial_\nu u(t) + \beta(t,.) u = 0 \text{ on } \Gamma.
\]
Here we use the following weak definition of the normal derivative.
Let $v \in H^1(\Omega)$ such that $\Delta v \in L^2(\Omega)$.
Let $h \in L^2(\Gamma, \d \sigma)$.  Then $\partial_\nu v = h$
by definition if
$\int_\Omega \nabla v \nabla w + \int_\Omega \Delta v w = \int_\Gamma h w \, \d \sigma$ for all $w \in H^1(\Omega)$.

Given $f \in L^2(0,T; H^1(\Omega)')$  and $u_0 \in L^2(\Omega)$, we denote by $u_\beta$ the solution of 
\begin{equation}\label{rob}
\left\{
\begin{aligned}
	&u \in \MR(H^1(\Omega),H^1(\Omega)') \text{ satisfying}\\
	&\dot u(t) + \A(t)u(t)=f(t) \quad t\text{-a.e.}\\
	&	u(0)=u_0.
\end{aligned}
\right.
\end{equation}
\begin{proposition}\label{proRob}
1) If $f \ge 0$ and $u_0 \ge 0$ then $u_\beta(t) \ge 0$ for all $t \in [0, T]$.\\
2) ({\it monotonicity}) If $\beta_1$ and $\beta_2$ are such that  $\beta_1(t, .)  \le \beta_2(t,.)$ for $t \in [0,T]$ and $u_0 \ge 0$, then 
$u_{\beta_2}(t) \le u_{\beta_1}(t)$ for all $t \in [0, T]$.
\end{proposition}
\begin{proof} Assertion 1) is a consequence of Proposition \ref{cor:positivity}. Assertion 2) follows from Proposition \ref{dom} since the forms with $\beta = \beta_1$ or $\beta= \beta_2$ satisfy the assumptions of this proposition. \end{proof}

An interesting consequence  of the previous proposition is that if $\beta \ge 0$ and $f = 0$, then 
\begin{equation}\label{gauss}
u_\beta(t)(x) \le C t^{-d/2} e^t \int_\Omega \exp[- c \frac{|x-y|^2}{t}] u_0(y) \ \d y.
\end{equation}
Here $C$ and $c$ are positive constants. The reason is that $u_\beta(t) \le u(t)$ where $u$ is the solution of \eqref{rob} with $\beta = 0$. In the latter case, the operator $A(t)$ is time-independent and coincides with the Neumann Laplacian. It is well known that the heat kernel of this operator has a Gaussian upper bound, see \cite{Dav89}, Chapter 3 or \cite{Ouh05}, Chapter 6.  This gives \eqref{gauss}. \\
Similarly, if one considers \eqref{rob} with $f= 0$ and initial data $u(s) = u_0 \ge 0$ at some $s \ge 0$,  then the estimate becomes
\begin{equation}
u_\beta(t)(x) \le C (t-s)^{-d/2} e^{(t-s)}  \int_\Omega \exp[- c \frac{|x-y|^2}{t-s}] u_0(y) \ \d y.
\end{equation}

\begin{remark}
We could replace in Proposition \ref{proRob} the Laplacian by an  elliptic operator  in divergence form as in example I). The statement and the  comments following this proposition hold in this setting with  the same proof. 
\end{remark}


\section{Appendix: Vector-valued 1-dimensional Sobo\-lev spaces}

We summarize some results on Hilbert space-valued Sobolev spaces . Given $u\in L^2(0,T;H)$ a function
 $\dot u \in L^2(0,T;H)$ is called \emph{the weak derivative} of $u$ if
\[
	-\int_0^T  u(s) \dot \varphi (s) \ \d s = \int_0^T  \dot u (s)\varphi (s) \ \d s
\]
for all $\varphi \in C^\infty_c(0,T)$. Thus we merely test with scalar-valued test functions $\varphi$ on $(0,T)$. It is clear that the weak derivative $\dot u$ of $u$ is unique whenever it exists. We let
\[
	H^1(0,T;H):=\{u\in L^2(0,T;H):u \text{  has  a weak derivative  }  \dot u \in L^2(0,T;H)\}.
\]
It is easy to see that $H^1(0,T;H)$ is a Hilbert space for the scalar product
\[
	(u\mid v)_{H^1(0,T;H)}:=\int^T_0 \Big[ (u(t)\mid v(t))_H + (\dot u (t)\mid \dot v (t))_H \Big] \ \d t.
\]
As in the scalar case \cite[Section 8.2]{Bre11} one shows the following.

\begin{proposition}\label{prop:hauptsatz}
{\rm a)}
 Let $u\in H^1(0,T;H)$. Then there exists a unique $w\in C([0,T];H)$ such that $u(t)=w(t)$ a.e.\ and
 \[
	 w(t)=w(0)+\int^t_0 \dot u (s) \ \d s.
 \]
{\rm b)} Conversely, if $w \in C([0,T];H),v\in L^2(0,T;H)$ such that $w(t)=w(0)+\int^t_0 v(s) \, \d s$, then $w \in H^1(0,T;H)$ and $\dot w = v$.
\end{proposition}

In the following we always identify $u\in H^1(0,T;H)$ with its unique continuous representative $w$ according to a).
We prove a vector-valued version of \cite[Proposition 9.3]{Bre11}.

\begin{proposition}\label{prop:H1_characterisation}
Let $u\in L^2(0,T;H)$. The following are equivalent:
\begin{enumerate}
\item[{\rm (i)}] $u\in H^1(0,T;H)$;
\item[{\rm (ii)}] there exists $C\ge 0$ such that for $0<c<d<T,\ \abs{h} < \min \{c,T-d\}$  one has  
	\[
		\int^d_c \norm{u(t+h)-u(t)}^2_H \ \d t \le C^2 \abs{h}^2 .
	\]
\end{enumerate}
In that case $C=\big(\int^b_a \norm{\dot u}^2_H \, \d t\big)^{1/2}$ is the optimal constant.
\end{proposition}

\begin{proof}
(i)$\Rightarrow$(ii).
Using at first the Cauchy-Schwarz inequality and then Fubini's Theorem we have for $h>0$
\begin{align*}
	\int^d_c \norm{u(t+h)&-u(t)}^2_H \ \d t\\
		&= \int^d_c \norm[\Big]{ \int^{t+h}_t \dot u (s) \ \d s }^2_H \ \d t \\
		&\le \int^d_c  \int^{t+h}_t \norm{\dot u(s)}^2 \ \d s \ \d t \cdot h\\
		&= h \Big[ \int^d_c \int^s_{s-h} \norm{\dot u(s)}^2_H \ \d t \ \d s 
				+ \int^{d+h}_d \int^s_d \norm{\dot u (s)}^2_H \ \d t \ \d s \Big] \\
		&\le h^2 \Big[ \int^d_c \norm{ \dot u (s) }^2_H \ \d s 
				+ \int^{d+h}_d \norm{ \dot u (s) }^2_H \ \d s  \Big]\\
		&\le h^2 \int^T_0 \norm{\dot u (s)}^2_H \ \d s.
\end{align*}
The proof for $h<0$ is similar.

(ii)$\Rightarrow$(i).
Let $v\in C^1_c((0,T);H)$. Let  $\supp v \subset [c,d] \subset [0,T]$ with $0<c<d<T$. 
Then for $\abs{ h } < \min \{c,T-d\}$,
\begin{align*}
\abs[\Big]{  \int_c^d (v(t+h)-v(t)\mid  u(t))_H \ \d t } 
	&= \abs[\Big]{ \int_c^d (v(t) \mid  u (t+h)-u(t))_H \ \d t   } \\
	&\le  C \abs{ h }  \norm{ v }_{L^2(0,T;H)} .
\end{align*}
It follows that
\[
	\abs[\Big]{ \int^T_0 (\dot v(t) \mid u (t))_H \ \d t } \le C \norm{v}_{L^2(0,T;H)}
\]
for all $v \in C^1_c (0,T;H)$. Since $C^1_c(0,T;H)$ is dense in $L^2(0,T;H)$ and since $L^2(0,T;H)^\prime = L^2(0,T;H)$ with respect to the natural duality, 
there exists $\dot u \in L^2(0,T;H)$ such that
\[
	- \int^T_0 (\dot v(t) \mid u(t))_H \ \d t 
		= \int^T_0 (v(t) \mid \dot u (t))_H \ \d t
\]
for all $v \in C^\infty_c(0,T;H)$. Given $\varphi \in C^\infty_c(0,T;H),v \in H$, choosing $v(t)=\varphi(t)v$ we deduce that
\[
	- \int^T_0 \dot \varphi (t)u(t) \ \d t = \int^T_0 \varphi (t) \dot u (t) \ \d t
\]
for all $\varphi \in C^\infty_c(0,T;H)$. Thus $u \in H^1(0,T;H)$.
\end{proof}

Now we come to the main point of this appendix which is needed in Section~\ref{section:invariance}.

\begin{theorem}\label{thm:appendix}
Let $S:H\to H$ be Lipschitz-continuous; i.e.
\[
\norm{S(x)-S(y)}_H \le L \norm{x-y}_H \quad (x,y \in H) ,
\]
where $L\ge 0$. Then $S \circ u \in H^1(0,T;H)$ for all $u \in H^1(0,T;H)$. Moreover, 
\[
	\norm{ (S\circ u)\dot{}\, }_{L^2(0,T;H)} \le L \norm{ \dot u }_{L^2(0,T;H)} .
\]
\end{theorem}

\begin{proof}
Let $u\in H^1(0,T;H)$. Then by Proposition \ref{prop:H1_characterisation}
\[
	\int^d_c\norm{u(t+h)-u(t)}^2_H \ \d t \le \int^T_0 \norm{\dot u (t)}^2 \ \d t \cdot \abs{h}^2
\]
whenever $0<c<d<T,\ \abs{h} < \min \{c,T-d\}$. Thus
\begin{align*}
	\int^d_c \norm{ S(u(t+h))-S(u(t)) }^2_H \ d t 
		&\le L^2 \int^d_c \norm{u(t+h)-u(t)}^2_H \ \d t \\
		&\le L^2 \int^T_0 \norm{\dot u (t)}^2_H \ \d t \cdot \abs{h}^2.
\end{align*}
Now the claim follows from Proposition~\ref{prop:H1_characterisation}.
\end{proof}


\noindent
\emph{Wolfgang Arendt}, \emph{Dominik Dier}, Institute of Applied Analysis, 
University of Ulm, 89069 Ulm, Germany,\\
\texttt{wolfgang.arendt@uni-ulm.de}, \texttt{dominik.dier@uni-ulm.de}

\quad\\
\noindent
\emph{El Maati Ouhabaz,} Institut de Math\'ematiques (IMB), Univ.\  Bordeaux, 351, cours de la Libération, 33405 Talence cedex, France,\\ 
\texttt{Elmaati.Ouhabaz@math.u-bordeaux1.fr}

\end{document}